\documentclass[12pt]{amsproc}

\usepackage{amsmath,amssymb,setspace,nicefrac}
\usepackage[active]{srcltx}
\usepackage[colorlinks, linkcolor=red, citecolor=blue, urlcolor=blue, hypertexnames=true]{hyperref}
\usepackage{amsrefs}

\allowdisplaybreaks
\usepackage[all]{xy}

\newcommand{\rk}{{\rm rk}}

\newcommand{\Q}{{\mathbb Q}}
\newcommand{\Z}{{\mathbb Z}}
\newcommand{\N}{{\mathbb N}}

\DeclareMathOperator{\tor}{tor}
\newcommand{\cC}{{\mathcal C}}
\newcommand{\cE}{{\mathcal E}}
\newcommand{\cJ}{{\mathcal J}}
\newcommand{\cM}{{\mathcal M}}
\newcommand{\cP}{{\mathcal P}}
\newcommand{\cN}{{\mathcal N}}
\newcommand{\cO}{{\mathcal O}}
\newcommand{\cW}{{\mathcal W}}
\newcommand{\cX}{{\mathcal X}}

\newtheorem{thm}{Theorem}[section]
\newtheorem{cor}[thm]{Corollary}
\newtheorem{lem}[thm]{Lemma}
\newtheorem{question}[thm]{Question}
\newtheorem{definition}[thm]{Definition}

\newtheorem{remark}[thm]{Remark}
\newtheorem{proposition}[thm]{Proposition}

\theoremstyle{definition}

\title[Entropy of algebraic actions]{Some remarks on the entropy for algebraic actions of amenable groups}

\author{Nhan-Phu Chung}
\address{Nhan-Phu Chung, MPI-MIS, Inselstra\ss e 22,
04103 Leipzig, Germany}
\email{chung@mis.mpg.de}

\author{Andreas Thom}
\address{Andreas Thom, Univ.\ Leipzig,
PF 100920, 04009 Leipzig , Germany}
\email{andreas.thom@math.uni-leipzig.de}

\begin{document}

\maketitle

\onehalfspace

\begin{abstract}
In this short note we study the entropy for algebraic actions of certain amenable groups. The possible values for this entropy are studied. Various fundamental results about certain classes of amenable groups are reproved using elementary arguments and the entropy invariant. We provide a natural decomposition of the entropy into summands contributed by individual primes and a summand corresponding to $\infty$. These results extend previous work by Lind and Ward on $p$-adic entropy.
\end{abstract}

\tableofcontents

\section{Introduction}

Let $\Gamma$ be an amenable group.
We denote by $\Z \Gamma$ the integral group ring of $\Gamma$. Let $\cM$ be a countable left $\Z\Gamma$-module. The Pontryagin dual of the underlying abelian group $\cM$ is denoted by $\widehat{\cM}:={\rm hom}_{\Z}(\cM,S^1)$. It is well-known that $\widehat{\cM}$ is a separable compact abelian group. Moreover, the group $\Gamma$ acts on $\widehat{\cM}$ by continuous automorphisms, preserving the Haar measure. Such an action $\Gamma \curvearrowright \widehat \cM$ is called \emph{algebraic action}. In this note we are interested in computations of the entropy of certain algebraic actions of the group $\Gamma$. The study of the entropy of algebraic actions has a long history, dating back to work of Yuzvinski\u{\i} \cite{MR0214726} in the case $\Gamma=\Z$ and later Lind-Schmidt-Ward in the case $\Gamma=\Z^d$, see \cites{MR1062797} and the references therein.
Even though the entropy theory for actions of general amenable groups was developed already in the 1980's \cites{ornsteinweiss}, it took some time until Deninger \cite{Deninger} initiated the study of entropy of algebraic actions of non-commutative amenable groups.

It is well-known that the topological entropy and the measure theoretic entropy with respect to the normalized Haar measure coincide in this case \cite[Theorem 2.2]{Deninger} -- and we denote the common value by
$$\rho(\cM) := h( \Gamma \curvearrowright \widehat \cM) \in [0,\infty].$$
The quantity $\rho(\cM)$ is called \emph{torsion} of $\cM$.

Important computations (with emphasis on principal algebraic actions) were obtained in work of Deninger and Deninger-Schmidt \cites{Deninger, Deninger-Schmidt}. These results were extended in work of Li and the second author \cite{lithom}. Indeed, $\rho(\cM)$ has been computed in the case when $\cM = \Z \Gamma/ \Z \Gamma f$ for any non-zero-divisor $f \in \Z \Gamma$ and more generally if $\cM$ is a $\Z \Gamma$-module of type FL, i.e.\ if $\cM$ admits a finite resolution
$$0 \to \Z \Gamma^{\oplus n_k} \to  \cdots \to \Z \Gamma^{\oplus n_1} \to \Z \Gamma^{\oplus n_0} \to \cM \to 0$$
consisting of finitely generated free modules. This terminology was introduced by Serre, FL comes from {\it finite libre}. It has been shown in \cite{lithom} that
\begin{equation*}
\rho(\Z \Gamma/ \Z \Gamma f) = \log \det\!{}_{\Gamma}(f),
\end{equation*} where $\det\!{}_{\Gamma}(f)$ denotes the Fuglede-Kadison determinant of $f$ -- seen as an element in the group von Neumann algebra $L(\Gamma)$ -- and more generally: \begin{equation} \label{lithom2}
\rho(\cM) = \rho^{(2)}(\cM),
\end{equation} where $\rho^{(2)}(\cM)$ denotes the $\ell^2$-torsion of the $\Z \Gamma$-module $\cM$. See \cite{lithom} for the necessary definitions and references to earlier results in this direction. Note that already $\rho^{(2)}(\cM) \geq 0$ is a non-trivial consequence of this equality.
It is known that $\log \det\!{}_{\Gamma}(f) < \infty$ if and only if $f$ is a non-zero-divisor in $\Z \Gamma$. Similarly, $\rho^{(2)}(\cM)< \infty$ if and only if the Euler characteristic $\chi(\cM) = \sum_{i=0}^k (-1)^i n_i \neq 0.$
The result in Equation \eqref{lithom2} is best possible in the sense that the $\ell^2$-torsion is only defined for $\Z \Gamma$-modules of type FL -- whereas $\rho(\cM)$ makes sense for any $\Z \Gamma$-module.

Topological entropy of algebraic action was also used by Elek in \cite{MR1952628} to define an invariant on finitely generated $\mathbb{F}_2\Gamma$-modules, where $\mathbb{F}_2$ is the field of two elements. In that paper, the integrality of values of such entropies was established when $\Gamma$ is a poly-$\Z$-group. In this note, we extend this result to torsionfree elementary amenable group and relate the computation of entropy (or topological mean dimension for that matter) of algebraic actions to notorious problems such as the integrality of $\ell^2$-Betti numbers and the Zero Divisor Conjecture.

At the same time, we want to study the abstract properties of the assignment $\cM \mapsto \rho(\cM)$ and compute $\rho(\cM)$ -- and hence the entropy for the corresponding algebraic action -- in many cases, which are not covered by the results in \cite{lithom}.
More precisely, motivated by the $p$-adic view point of Lind-Ward \cite{MR961739} on computing entropy of solenoids for $\Z$-actions, we extended their results to all amenable groups satisfying the Zero Divisor Conjecture. This is a new result even for the $\Z^2$-case. Solenoid entropy also has been used to compute the growth of order of the first homology of the $r$-fold cyclic covering branched over a knot \cite{Noguchi}. On the other side, using the Shnirelman integral, Everest et al.\ also established the $p$-adic Mahler measure and its relations to the canonical height of some elliptic curve \cite{Einsiedler, EEW,EF, Everest,EW}. The $p$-adic Mahler measure is also used to illustrate the $p$-component of torsion numbers in knot theory \cite{SW}.

\section{The von Neumann rank and its properties}
In order to complete the picture, we have to note that $\cM \mapsto \rho(\cM)$ behaves very much like a secondary invariant on the category of $\Z \Gamma$-modules. But what is the primary invariant? Let $\Gamma$ be an amenable group and let $\cM$ be a $\Z \Gamma$-module. We set
$${\rm rk}(\cM):= \dim_{L(\Gamma)} \left(L(\Gamma) \otimes_{\Z \Gamma} \cM\right) \in [0,\infty],$$
where $L(\Gamma)$ is the group von Neumann algebra and $\dim_{L(\Gamma)}$ denotes L\"uck's extension of the Murray-von Neumann dimension function, see \cite[Section 6.1]{Luck} for details.
Let us call $\rk(\cM)$ the {\it von Neumann rank} of $\cM$.
We summarize the properties of the assignment $\cM \mapsto {\rm rk}(\cM)$ below:
\begin{enumerate}
\item \label{yuzrank}
If $0\to \cM' \to \cM \to \cM'' \to 0$ is exact, then
$\rk(\cM) = \rk(\cM') + \rk(\cM'')$.
\item If $\cM = {\rm colim}_n \cM_n$ and $\cM'_n := \cM_n/ \left({\cup_{m \geq n} \ker(\cM_n \to \cM_m)} \right)$, then $$\rk(\cM) = \sup_n \rk(\cM'_n) = \lim_{n \to \infty} \rk(\cM'_n).$$
In particular, if $\cM= \oplus_n \cM_n$, then $\rk(\cM) = \sum_n \rk(\cM_n)$.
\item \label{normal}
If $\cP = (\Z \Gamma^{\oplus n})p$ is a finitely projective module and $p^2=p \in M_n(\Z \Gamma)$, then $\rk(\cP)= {\rm tr}(p) \in \N$. In particular, $\rk(\Z \Gamma)=1$.
\item If $\cM \subset \Z \Gamma$ and $\rk(\cM)=0$, then $\cM=0$.
\end{enumerate}

If $\cM$ is of type FL and
$$0 \to \Z \Gamma^{\oplus n_k} \to  \cdots \to \Z \Gamma^{\oplus n_1} \to \Z \Gamma^{\oplus n_0} \to \cM \to 0$$
is a finite resolution of $\cM$, then Properties \eqref{yuzrank} and \eqref{normal} from above imply that
$$\rk(\cM) = \sum_i (-1)^i n_i =\chi(\cM).$$
In particular, $\chi(\cM) \geq 0$ which is already a non-trivial result. Very recently, it has been shown that the von Neumann rank of $\cM$ has also a natural definition in terms of the topological dynamical system $\Gamma \curvearrowright \widehat{\cM}$. It has been shown by Li-Liang \cite{liprep} that $$\rk(\cM) = \dim_{L(\Gamma)}\left( L(\Gamma) \otimes_{\Z \Gamma} \cM\right)={\rm mdim}(\Gamma \curvearrowright \widehat{\cM}),$$ where mdim denotes the mean topological dimension as defined and studied by Gromov \cite{gromov} and Lindenstrauss-Weiss \cite{LW}.

We want to show below that $\cM \to \rho(\cM)$ enjoys similar formal properties and in fact is useful to study the properties of group rings with coefficients in finite fields instead of $\Z$. We can summarize the results in Proposition \ref{colim}, Corollary  \ref{sum}, and Lemma \ref{posent} as follows:
\begin{enumerate}
\item
If $0\to \cM' \to \cM \to \cM'' \to 0$ is exact, then
$\rho(\cM) = \rho(\cM') + \rho(\cM'')$.
\item If $\cM = {\rm colim}_n \cM_n$ and $\cM'_n := \cM_n/ \left({\cup_{m \geq n} \ker(\cM_n \to \cM_m)} \right)$, then $$\rho(\cM) = \sup_n \rho(\cM'_n) = \lim_{n \to \infty} \rho(\cM'_n).$$
In particular, if $\cM= \oplus_n \cM_n$, then $\rho(\cM) = \sum_n \rho(\cM_n)$.
\item \label{normal}
Let $k$ be a finite abelian group. Then, $\rho(k \Gamma)= \log |k|$.
\item If $\cM \subset k \Gamma$ for some finite abelian group $k$ and $\rho(\cM)=0$, then $\cM=0$.
\end{enumerate}

But let us return to coefficients in $\Z$ for a moment. Some tantalizing open problems can be reformulated in terms of the rank-function.

\begin{question} \label{integralrank}
Let $\Gamma$ be a torsionfree amenable group. Is it true that $\rk(\cM) \in \N \cup \{\infty\}$ for every $\Z \Gamma$-module?
\end{question}

As it turns out, a positive answer to this question is equivalent to a positive solution of the Zero Divisor Conjecture for $\Z \Gamma$, see Remark \ref{liZ}. Since $\Gamma$ is amenable, it is also the same as the Atiyah Conjecture, i.e.\ to ask for integrality of $\ell^2$-Betti numbers, see \cite{Schick01} and \cite[Section 10.16]{Luck}.

In order to illustrate the strength of the rank-function, we want to start out with the following easy proof of the Zero Divisor Conjecture for amenable groups of finite cohomological dimension, whose integral group ring is noetherian. The only groups which are known to have a noetherian group ring are polycyclic-by-finite.
If a group $\Gamma$ is polycylic-by-finite, then there exists $\{1\}=\Gamma_0\lhd \Gamma_1...\lhd \Gamma_m=\Gamma$ such that $\Gamma/\Gamma_{m-1}$ is finite and $\Gamma_k/\Gamma_{k-1}\simeq \Z$ for every $k=1,...,m-1$, see \cite[Lemma 10.2.5]{passman}.
If such a group is torsionfree then by Theorem 1 and Proposition 6 of \cite{serre}, one has ${\rm cd}(\Gamma)={\rm cd}(\Gamma_{m-1})<\infty$, where ${\rm cd}$ denotes the cohomological dimension of $\Gamma$, see Definition \ref{cohodef}.

The proof of the following theorem takes its inspiration from \cite{MR0422327}, where the Goldie rank was used in a similar (somehow more involved) argument instead of the von Neumann rank.

\begin{thm}[Farkas-Snider] \label{zdcZ}
Let $\Gamma$ be a amenable group such that
\begin{enumerate}
\item $\Gamma$ has finite cohomological dimension, and
\item the group ring $\Z \Gamma$ is noetherian.
\end{enumerate}
Then, $\Z \Gamma$ does not contain any non-trivial zero divisors.
\end{thm}
\begin{proof}
Let $f \in \Z \Gamma$ be arbitrary.
The $\Z \Gamma$-module $\ker(f) := \{h \in \Z \Gamma \mid hf=0\} \subset \Z \Gamma$ is finitely generated as $\Z \Gamma$ is assumed to be noetherian. Assuming also that the cohomological dimension of $\Gamma$ is finite, there exists a finite resolution of $\ker(f)$ by finitely generated projective modules
$$0 \to \cP_n \to \cP_{n-1} \to \cdots \to \cP_0 \to \ker(f) \to 0.$$
Hence,
${\rm rk}(\ker(f)) = \sum_{k=0}^n (-1)^k{\rm rk}(\cP_k) \in \Z$. On the other hand $\ker(f) \subset \Z \Gamma$, so clearly ${\rm rk}(\ker(f)) \in [0,1]$.
Hence, either $\rk(\ker(f))=0$ and thus $\ker(f)=0$, or $\rk(\ker(f))= 1$ and hence $\rk({\rm im}(f))=0$ for ${\rm im}(f) =\Z \Gamma \cdot f \subset \Z \Gamma$. In the second case, we obtain $f=0$. This proves the claim.
\end{proof}

Note that the Zero Divisor Conjecture with coefficients in a finite field implies the Zero Divisor Conjecture with coefficients in $\Z$ by an easy reduction argument -- but not conversely. Hence, it would be desirable to extend the argument above to cover group rings with more general coefficients, in particular, with coefficients in a finite field. Since $\rk(k \Gamma)=0$ for any finite ring $k$, the role of the von Neumann rank has to be played by the torsion. It is clear that we first have to study the abstract properties of $\cM \mapsto \rho(\cM)$ more systematically.

\section{Properties of the torsion invariant}

In this short section, we want to mention some general results that are useful in the computation of the torsion in various situations. First of all, if $$ 0 \to \cM' \to \cM \to \cM'' \to 0$$
is an exact sequence of $\Z \Gamma$-modules, then Yuzvinski\u{\i}'s additivity formula states
\begin{equation} \label{yuz}
\rho(\cM) = \rho(\cM') + \rho(\cM'').
\end{equation}
A first version of this equation was proved for $\Z$-actions by Yuzvinski\u{\i} in \cite{MR0214726}. In this most general setup, the formula above has been established in seminal work of Li \cite{li}.

Let $\Gamma$ be an amenable group and $M$ be a countable $\Z\Gamma$-module. For any nonempty finite subset $\mathcal E$ of $M$, the function $F\mapsto \log |\sum_{s\in F}s^{-1}{\mathcal E}|$ defined on the set of nonempty finite subsets of $\Gamma$ satisfies the conditions of Ornstein-Weiss lemma \cite[Theorem 6.1]{LW}. Then $ \frac{\log |\sum_{s\in F}s^{-1}{\mathcal E}|}{|F|}$ converges to a real number $c$, denoted by $\lim_F \frac{\log |\sum_{s\in F}s^{-1}{\mathcal E}|}{|F|}$,
when $F$ becomes more and more left invariant in the sense that for any $\varepsilon>0$, there exist a nonempty finite subset $K$ of $\Gamma$ and $\delta>0$ such that for any
nonempty finite subset $F$ of $\Gamma$ satisfying $|KF\setminus F|\le \delta |F|$ one has
$|\frac{\log |\sum_{s\in F}s^{-1}\mathcal E|}{|F|}-c|<\varepsilon.$

The following result of Peters, Theorem 6 in \cite{pet}, gives an intrinsic description of the entropy of an algebraic action.

\begin{thm}[Peters] \label{peters} Let $\Gamma$ be an amenable group and let $\cM$ be a countable $\Z \Gamma$-module.
Then
$$ \rho(\cM)=\sup_{\mathcal E} \lim_F \frac{\log |\sum_{s\in F}s^{-1}{\mathcal E}|}{|F|},$$
where ${\mathcal E}$ ranges over all nonempty finite subsets of $\cM$ and $F \subset \Gamma$ becomes more and more invariant.
\end{thm}
Note that Theorem~\ref{peters} was stated and proved only for the case $\Gamma=\Z$ in \cite{pet}, but the proof there works for general countable amenable groups.
\begin{proposition} \label{colim}
Let $\cM_0 \to \cM_1 \to \cdots$ be a sequence of modules and $\cM = {\rm colim}_n \cM_n$. Then,
$$\rho(\cM) = \lim_{n \to \infty} \rho(\cM'_n) \quad \mbox{with} \quad \cM'_n := \frac{\cM_n}{\bigcup_{m \geq n} \ker(\cM_n \to \cM_m)}.$$ In particular, if the maps $\cM_n \to \cM_{n+1}$ are injective for $n$ large enough, then $$\rho(\cM)= \lim_{n \to \infty} \rho(\cM_n).$$
\end{proposition}
\begin{proof}
It is clear that $\cM_n' \subset \cM$ for all $n$ and $\cup_n \cM_n' = \cM$. Now, Theorem \ref{peters} implies the claim.
\end{proof}

\begin{cor} \label{sum}
Let $\cM_0,\cM_1,\dots$ be a sequence of $\Z\Gamma$-modules. Then,
$$\rho\left( \bigoplus_{n \geq 0} \cM_n \right) = \sum_{n=0}^{\infty} \rho(\cM_n).$$
\end{cor}
\begin{proof}
This is an easy consequence of Theorem \ref{peters}.
\end{proof}

\section{Values for the entropy of algebraic actions}
\label{hli}
In this section, we study the restrictions on the possible values for the torsion of $\Z \Gamma$-modules, provided that $\Gamma$ is torsionfree and satisfies some form of the Zero Divisor Conjecture. Lemmas \ref{posent} and \ref{L-k zero divsior} are due to Hanfeng Li. We have the kind permission to include them in this paper.

\begin{lem} \label{posent}
Let $k$ be a finite abelian group. If $\cM \subset  k \Gamma$ is a $\Z \Gamma$-module, then either $\cM = \{0\}$ or  $\rho(\cM) >0$. Moreover if $f \in \Z \Gamma$ has support $S \subset \Gamma$, then
$$\rho( k \Gamma \cdot f) \geq \frac{\log |k|}{|S|^2}.$$
\end{lem}
\begin{proof}
Denote by $S$ the support of $f$. Let $F$ be a nonempty finite subset
of $\Gamma$. Take a maximal subset $W$ of $F$ subject to the condition that the sets $S^{-1}s$ are pairwise disjoint for $s\in W$. Then $F\subseteq \bigcup_{s\in W}SS^{-1}s$, and hence $|W|\ge |F|/|SS^{-1}|$. Set $\cE= k \cdot f$. Since the supports of $s^{-1}f$ are pairwise disjoint for $s\in W$, we
have
$$ |\sum_{s\in F}s^{-1}\cE|\ge |\sum_{s\in W}s^{-1}\cE|=|k|^{|W|}.$$
From Peters' Theorem (i.e.\ Theorem \ref{peters}) we get
$$\rho(k \Gamma \cdot f)\ge \lim_F\frac{\log |\sum_{s\in F}s^{-1}\cE|}{|F|}\ge \liminf_F\frac{\log |k|^{|W|}}{|F|} \geq \frac{ \log|k|}{|SS^{-1}|} \geq \frac{\log |k|}{ |S|^2}>0.$$ This finishes the proof of the lemma.
\end{proof}

We are now ready to extend the argument used in Theorem \ref{zdcZ} to cover also the case of coefficients in a finite field. Again, it applies to the class of torsionfree polycyclic-by-finite groups, see the work of Farkas-Snider \cite{MR0422327} and the references therein.

\begin{thm} \label{zerofin}
Let $k$ be a finite field. Let $\Gamma$ be a amenable group such that
\begin{enumerate}
\item $\Gamma$ has finite cohomological dimension,
\item the group ring $k \Gamma$ is noetherian, and
\item $K_0(k \Gamma)=\Z$.
\end{enumerate}
Then, $k \Gamma$ does not contain any non-trivial zero divisors.
\end{thm}
\begin{proof}
Let $f \in k \Gamma$ be arbitrary.
The $k \Gamma$-module $\ker(f) := \{h \in k \Gamma \mid hf=0\} \subset k \Gamma$ is finitely generated as $k \Gamma$ is noetherian. As the cohomological dimension of $\Gamma$ is finite, there exists a finite resolution of $\ker(f)$ by finitely generated projective $k \Gamma$-modules
$$0 \to \cP_n \to \cP_{n-1} \to \cdots \to \cP_0 \to \ker(f) \to 0.$$
If $K_0(k \Gamma)=\Z$, then $\cP_l$ is stably free for each $0 \leq l \leq n$. We conclude from this $\rho(\cP_l) \in \log|k| \cdot \N$.
Hence,
${\rm \rho}(\ker(f)) = \sum_{l=0}^n (-1)^l \rho(\cP_l) \in \log|k| \cdot \Z$. On the other hand $\ker(f) \subset k \Gamma$, and so $\rho(\ker(f)) \in [0,\log|k|]$.
Hence, either $\rho(\ker(f))=0$ and thus $\ker(f)=0$ by Lemma \ref{posent}, or $\rho(\ker(f))= \log|k|$ and hence $\rho({\rm im}(f))=0$. Since ${\rm im}(f) = k \Gamma \cdot f \subset k \Gamma$, Lemma \ref{posent} implies $f=0$ in the second case. This proves the claim.
\end{proof}

The preceding result is well-known but was not covered by the results in \cite{MR0422327}. However, the work of Farkas-Snider has been extended to a more general situation as follows:

\begin{thm}[Kropholler-Linnell-Moody, Theorem 1.4 in \cite{MR964842}] \label{eleman}
Let $\Gamma$ be a torsionfree elementary amenable group and let $k$ be a skew field. The group ring $k \Gamma$ does not contain zero-divisors.
\end{thm}

The following lemma provides the link between the Zero Divisor Conjecture and the quantization of the values for the entropy.

\begin{lem} \label{L-k zero divsior}
Let $k$ be a finite field and $\Gamma$ be a countable amenable group. Then the following are equivalent:
\begin{enumerate}
\item For any $k\Gamma$-module $\cM$, treating $\cM$ as  a left $\Z\Gamma$-module, one has that $\rho({\cM})/\log |k|$ is either an integer or $\infty$;

\item For any finitely generated $k\Gamma$-module $\cM$, treating $\cM$ as  a $\Z\Gamma$-module, one has that $\rho({\cM})/\log |k|$ is an integer;

\item For any nonzero left ideal $\cJ$ of $k\Gamma$, treating $k\Gamma/\cJ$ as  a $\Z\Gamma$-module, one has  $\rho({k\Gamma/\cJ})=0$;

\item For any nonzero $f$ in $k\Gamma$, treating $k\Gamma/k\Gamma f$ as  a $\Z\Gamma$-module, one has  $\rho({k\Gamma/k\Gamma f})=0$;

\item $k \Gamma$ contains no zero-divisors.
\end{enumerate}
\end{lem}
\begin{proof}
(2)$\Rightarrow$(1): Proposition \ref{colim} says that for any $\Z\Gamma$-module $\cM$, which is the union of an increasing sequence $\{\cM_n\}_{n\in \N}$ of sub-$\Z\Gamma$-modules, one has
$\rho({\cM})=\lim_{n\to \infty}\rho({\cM_n})$.

(1)$\Rightarrow$(2): For any $n\in \N$, the natural action of $\Gamma$ on $\widehat{(k\Gamma)^n}$ is the full Bernoulli shift with $|k|^n$ symbols, and thus has topological entropy $n\log |k|$. It follows that for any finitely generated left $k\Gamma$-module $\cM$, one has $\rho({\cM})<\infty$.

(3)$\Rightarrow$(2): From the Yuzvinski\u{\i} additivity formula, for any $\Z\Gamma$-module $\cM$ and any sub-$\Z\Gamma$-module $\cM'$, one has
$\rho({\cM})=\rho({\cM'})+\rho({\cM/\cM'})$. Then by induction on the number of generators we get the conclusion.

(2)$\Rightarrow$(3): Let $\cJ$ be a non-zero left ideal of $k\Gamma$. Take a non-zero $f$ in $\cJ$. Since $k \Gamma f \subset \cJ$, it follows from Lemma \ref{posent}, that $\rho(\cJ)>0$.
From the Yuzvinski\u{\i} additivity formula, we have
$$ \rho({\cJ})+\rho({k\Gamma/\cJ})= \rho({k\Gamma})=\log |k|.$$ Hence, we conclude $\rho({k\Gamma/\cJ})=0$.

(4)$\Rightarrow$(3): Let $\cJ$ be a non-zero left ideal of $k\Gamma$. Take a nonzero $f$ in $\cJ$. Then $\rho({k\Gamma/k\Gamma f})\ge \rho({k\Gamma/\cJ})$.

(3)$\Rightarrow$(4) is trivial.

(4)$\Rightarrow$(5): Let $f\in k\Gamma$ be non-zero. Denote  by $\cW$ the left ideal of $k\Gamma$ consisting of $g\in k\Gamma$ satisfying $gf=0$.
From the Yuzvinski\u{\i} additivity formula we have
$$ \rho({\cW})+\rho({k\Gamma f})=\rho({k\Gamma})<\infty$$
and
$$ \rho({k\Gamma f})+\rho({k\Gamma/k\Gamma f})=\rho({k\Gamma})<\infty.$$
Thus $\rho({\cW})=\rho({k\Gamma/k\Gamma f})$. Assuming (4), we get $\rho({\cW})=0$, and hence $\cW=0$.

(5)$\Rightarrow$(4): Assuming (5), we get $\rho({k\Gamma/k\Gamma f})=0$.
\end{proof}

\begin{remark} \label{liZ}
It is clear that an identical argument (using the von Neumann rank instead of the torsion) shows that Question \ref{integralrank} is indeed equivalent to the Zero Divisor Conjecture with coefficients in $\Z$.
\end{remark}

Let us record the following corollary of Lemma \ref{L-k zero divsior} and Theorem \ref{eleman}.

\begin{cor} \label{C-elek} Let $\Gamma$ be a torsionfree elementary amenable group and let $\cM$ be a left $\Z \Gamma$-module such that $p \cM = \{pm \mid m \in \cM\}=0$ for some prime $p$. Then, $\rho({\cM} )$ is a multiple of $\log(p)$ or $\infty$.
\end{cor}

We conjecture that the preceding corollary holds for all torsionfree amenable groups -- in view of Lemma \ref{L-k zero divsior} this is equivalent to the Zero Divisor Conjecture.
Elek obtained the preceding result for poly-$\Z$ groups in \cite[Proposition 11.2]{MR1952628}, compare also \cite{Elek03a} and Theorem \ref{zerofin}.

\begin{question} \label{conj}
Let $\Gamma$ be a torsionfree amenable group. Is there a constant $c>0$, such that
$\rho(\cM)>c$ for every left $\Z \Gamma$-module $\cM$ unless $\rho({\cM})=0$?
\end{question}

Note that a positive answer to this question has some formal analogy with Question \ref{integralrank} and hence Atiyah's conjecture about the integrality of $\ell^2$-Betti numbers for torsionfree groups, see for example \cite{Schick01} for details. However, for $G=\Z$, Question \ref{conj} is equivalent to Lehmer's famous question about the minimal Mahler measure of a polynomial with integer coefficients and no cyclotomic factors. The case $G=\Z^d$ can be reduced to the case $\Z$ by results of Lawton, \cite{lawton}.

\section{The decomposition into primary components}

From now on we write $\mu_p := \Z[\frac1p]/\Z$. Note that there is an exact sequence
$$0 \to \Z \to \Z[\nicefrac1p] \to \mu_p \to 0,$$
which induces an exact sequence
\begin{equation} \label{seq1}
0 \to {\rm Tor}(\mu_p,\cM) \to \cM \to \Z[\nicefrac1p] \otimes_{\Z} \cM \to \mu_p \otimes_{\Z} \cM \to 0,
\end{equation}
for any abelian group $\cM$. Here, we have ${\rm Tor}(\mu_p,\cM) = \{x \in \cM \mid \exists k \in \N\ p^k x = 0 \}.$ If $\cM$ is a $\Z \Gamma$-module
satisfying $\rho(\cM)< \infty$, we set $$\rho_p(\cM):= \rho({\rm Tor}(\mu_p,\cM)) - \rho(\mu_p \otimes_{\Z} \cM).$$
If $\rho_p(\cM)$ is defined, we get from Yuzvinski\u{\i}'s additivity formula \eqref{yuz} and the exact sequence in \eqref{seq1} that
\begin{equation} \label{eqp}
\rho(\cM) = \rho(\Z[\nicefrac1p] \otimes_{\Z} \cM) + \rho_p(\cM).
\end{equation}
\begin{lem} \label{positive}
For any $\Z \Gamma$-module $\cM$ such that $\rho_p(\cM)$ is defined, we have $\rho_p(\cM) \geq 0.$ Moreover, if $\cM$ does not contain $p$-torsion, then $\rho_p(\cM)=0$.
\end{lem}
\begin{proof}
We have $$\Z[\nicefrac1p] \otimes_{\Z} \cM = {\rm colim}_n (\cM \stackrel{p}{\to} \cM \stackrel{p}{\to} \cdots ).$$ Now, Proposition \ref{colim} implies that $\rho(\Z[\nicefrac1p] \otimes_{\Z} \cM) \leq \rho(\cM)$ and hence $\rho_p(\cM)\geq0$ from \eqref{eqp}. If $\cM$ does not contain $p$-torsion, then ${\rm Tor}(\mu_p,\cM)=0$ and hence $\rho_p(\cM)=0$. This finishes the proof.
\end{proof}

In anology to the finite places, we set $\rho_{\infty}(\cM) = \rho(\Q \otimes_{\Z} \cM)$ for any $\Z \Gamma$-module $\cM$. Our main observation is now:

\begin{thm} \label{main1}
Let $\cM$ be a $\Z \Gamma$-module with finite torsion. Then, we have
\begin{equation}
\rho(\cM) = \rho_{\infty}(\cM) + \sum_{p} \rho_p(\cM).
\end{equation}
Moreover, for any exact sequence $0 \to \cM' \to \cM \to \cM'' \to 0$ of $\Z \Gamma$-modules with finite torsion, we have $\rho_p(\cM)= \rho_p(\cM') + \rho_p(\cM'')$ for any prime $p$, and $\rho_\infty(\cM)= \rho_\infty(\cM') + \rho_\infty(\cM'')$.
\end{thm}
\begin{proof}
From the sequence
$$0 \to \Z \to \Q \to \bigoplus_p \mu_p \to 0$$ we obtain the exact sequence
$$0 \to \bigoplus_p {\rm Tor}(\mu_p,\cM) \to \cM \to \Q \otimes_{\Z} \cM \to \bigoplus_p \left(\mu_p \otimes_{\Z} \cM  \right)\to 0.$$
Using now Yuzvinski\u{\i}'s additivity formula \eqref{yuz} and Corollary \ref{sum}, we obtain the desired formula.
Since $\cM \mapsto \Z[\nicefrac1p] \otimes_{\Z} \cM$ is exact, $\cM \mapsto \rho(\Z[\nicefrac1p] \otimes_{\Z} \cM)$ is additive on exact sequences by \eqref{yuz}. Now, Equation \eqref{eqp} implies that $\cM \mapsto \rho_p(\cM)$ must be additive as well. Again, since $\cM \to \Q \otimes_{\Z} \cM$ is exact, \eqref{yuz} implies that $\cM \to \rho_{\infty}(\cM)$ is additive. This proves the theorem.
\end{proof}

Naturally, the question arises whether the quantities $\rho_p$ and $\rho_{\infty}$ can be computed in special cases such as $\cM= \Z \Gamma/ \Z \Gamma f$ for a non-zero-divisor $f \in \Z \Gamma$. For $f \in \Z \Gamma$, we define $|f|_p$ to be the largest integer $k \in \N$, such that $p^{-k}f \in \Z \Gamma$. We set for any prime number $p$ or $\infty$,
$$L_p(f) := \rho_p(\Z \Gamma/ \Z \Gamma f).$$

Our next aim is to establish a computation of $L_p(f)$ for torsionfree elementary amenable groups.

\begin{thm}
Let $\Gamma$ be a torsionfree elementary amenable group and let $f \in \Z \Gamma$ be non-zero. Then, for any prime number $p$, we get:
\begin{equation} \label{compel}
\rho_p(\Z \Gamma/ \Z \Gamma f) = |f|_p \cdot \log(p).
\end{equation}
\end{thm}
\begin{proof}
Under the assumptions on $\Gamma$, $k \Gamma$ does not have any non-zero-divisors if $k$ is an integral domain, see \cite{MR964842}. This implies that $|fg|_p = |f|_p + |g|_p$ for any non-zero $f,g \in \Z \Gamma$. In order to show this, we may assume that $|f|_p = |g|_p =0$. Indeed, if $|fg|_p >0$, then the reduction of $f$ mod $p$ would be a non-trivial zero-divisor in $(\Z/p\Z)\Gamma$. Hence, we get $|fg|_p \leq |f|_p + |g|_p$ for all non-zero $f,g \in \Z \Gamma$. The other inequality is obvious.

In this situation, there is an exact sequence
$$
0 \to \Z \Gamma/ \Z\Gamma g \to \Z \Gamma/ \Z\Gamma gf \to \Z \Gamma/ \Z\Gamma f \to 0
$$
for any non-zero $f,g \in \Z \Gamma$. Hence, Yuzvinski\u{\i}'s additivity formula implies $L_p(fg)= L_p(f) + L_p(g)$ for any non-zero $f,g \in \Z \Gamma$. Since $\mu_p \otimes_{\Z} \cM = \cM$ and $\Z[\nicefrac1p] \otimes_{\Z} \cM =\{0\}$ for $\cM := (\Z/p \Z) \Gamma$, we get $L_p(p)= \log(p)$.

Thus -- knowing that both sides of the Equation \eqref{compel} are additive --  we may assume that $|f|_p=0$ in order to establish Equation \eqref{compel}. We claim  that in this case, $\Z \Gamma / \Z \Gamma f$ does not contain any $p$-torsion. Indeed, if $p h = gf$ for some non-zero $g,h \in \Z \Gamma$, then the image of $g$ in $(\Z/p \Z) \Gamma$ is a zero-divisor. Hence, $|g|_p \geq 1$ and $h$ represents zero in $\Z \Gamma/ \Z \Gamma f$ as $h= (g/p)f$ and $g/p \in \Z \Gamma$.

However, if $\Z \Gamma/ \Z \Gamma f$ does not contain $p$-torsion, then $\rho_p(\Z \Gamma/ \Z \Gamma f)=0$ by Lemma \ref{positive}. This proves the claim.
\end{proof}

More generally, we can interpret $L_p(f) = \rho_p(\Z \Gamma/ \Z \Gamma f)$ as an analogue of the Fredholm index of an operator. \begin{proposition} Let $\Gamma$ be an amenable group and $f \in \Z \Gamma$ be a non-zero-divisor. Then, we have
$$L_p(f) = \rho(\ker(f \colon \mu_p \Gamma \to \mu_p \Gamma)) - \rho({\rm coker}(f\colon \mu_p \Gamma \to \mu_p \Gamma)).$$
\end{proposition}
\begin{proof}
The snake lemma applied to the diagram
$$
\xymatrix{
0 \ar[r] & \Z \Gamma \ar[d]^{f} \ar[r] &\Z[\nicefrac 1p] \Gamma \ar[d]^{f} \ar[r]& \mu_p\Gamma \ar[d]^{f} \ar[r] & 0 \\
0 \ar[r] & \Z \Gamma  \ar[r] &\Z[\nicefrac 1p] \Gamma  \ar[r]& \mu_p\Gamma  \ar[r] & 0 }
$$
yields an exact sequence
$$0 \to \ker(f \colon \mu_p \Gamma \to \mu_p \Gamma) \to \Z \Gamma/ \Z \Gamma f \to \Z[\nicefrac 1p] \otimes_{\Z} \Z \Gamma/ \Z \Gamma f \to {\rm coker}(f\colon \mu_p \Gamma \to \mu_p \Gamma) \to 0.$$ Hence, we get
$$L_p(f)= \rho_p(\Z \Gamma/ \Z \Gamma f) = \rho(\Z \Gamma/ \Z \Gamma f) - \rho(\Z[\nicefrac 1p] \otimes_{\Z} \Z \Gamma/ \Z \Gamma f) =$$
$$=\rho(\ker(f \colon \mu_p \Gamma \to \mu_p \Gamma)) - \rho({\rm coker}(f\colon \mu_p \Gamma \to \mu_p \Gamma)).$$ This finishes the proof.
\end{proof}

\section{Computations for the entropy of $\Q \Gamma$-modules}

Let denote by $s(\Z \Gamma)$ the set of non-zero-divisors in the ring $\Z \Gamma$ and similarly by $s(\Q \Gamma)$ the set of non-zero-divisors in $\Q \Gamma$. Clearly, $s(\Z \Gamma)$ and $s(\Q \Gamma)$ are monoids. For any $f,g \in s(\Z \Gamma)$ we have
$$
0 \to \Z \Gamma/ \Z\Gamma g \to \Z \Gamma/ \Z\Gamma gf \to \Z \Gamma/ \Z\Gamma f \to 0
$$
and using Theorem \ref{main1}, we get $L_p(fg)=L_p(f) + L_p(g)$, for all prime numbers $p$ and $\infty$. It is obvious that $L_p$ has a unique extension to $s(\Q \Gamma)$, satisfying the same additivity property.

\begin{thm}
\label{T-solenoid entropy}
Let $\Gamma$ be an amenable group and $f \in \Q \Gamma$ be a non zero divisor. Then,
$$\rho(\Q \Gamma/ \Q \Gamma f) = \log \det\!{}_{\Gamma}(f) - \sum_{p} L_p(f)$$
\end{thm}
\begin{proof}
Let $n \in \N$ be such that $h:=nf \in \Z \Gamma$. Now, $\Q \Gamma/ \Q \Gamma f= \Q \Gamma/ \Q \Gamma h$ and $\Q\Gamma/\Q\Gamma f = \Q \otimes_{\Z}\left( \Z \Gamma/\Z \Gamma h \right)$. Hence, we get
$$\rho(\Q \Gamma/ \Q \Gamma f) = \log \det\!{}_{\Gamma}(nf) - \sum_{p} L_p(nf)$$
by Theorem \ref{main1} applied to $\cM=\Z \Gamma/\Z \Gamma h$ and \cite[Theorem 1.2]{lithom}. Now, clearly $\log n = \sum_p L_p(n)$, so that we get the desired conclusion.
\end{proof}

\begin{remark}
Theorem \ref{T-solenoid entropy} also holds for every non-zero-divisor $f\in M_n(\Q\Gamma)$ and $n\in\N$.
\end{remark}

\begin{cor}
Let $\Gamma$ be a torsionfree elementary amenable group and $f = \sum_{\gamma} f_{\gamma} \gamma \in \Q \Gamma$ non-zero. Let us write $f_{\gamma} = \frac{a_{\gamma}}{b_{\gamma}}$ with $a_{\gamma}, b_{\gamma} \in \Z$ coprime. Then
$$\rho(\Q \Gamma/\Q \Gamma f) = \log \det\!{}_{\Gamma}(f) - \log(\gcd\{a_\gamma \mid \gamma \in \Gamma\}) + \log({\rm lcm}\{b_\gamma \mid \gamma \in \Gamma\}).$$
\end{cor}
\begin{proof}
If $\Gamma$ is torsionfree elementary amenable and $f \in \Q \Gamma$ non-zero, then $L_p(f) = |f|_p \cdot \log(p)$, i.e.\! for $f = \sum_{\gamma} f_\gamma \gamma$ we have $L_p(f)= \min\{ |f_\gamma|_p \mid \gamma \in \Gamma \} \cdot \log(p)$, where $|.|_p$ denotes the usual $p$-adic valuation. Thus, we obtain
\begin{equation}
\rho(\Q \Gamma/ \Q \Gamma f) = \log \det\!{}_{\Gamma}(f) - \sum_p |f|_p \cdot \log(p).
\end{equation}
This proves the claim.
\end{proof}

Let us mention the following computation, which can be reduced to the previous corollary.

\begin{lem}[Lind-Ward, see \cite{MR961739}] Let $a \in GL_n\Q$ and consider $t-a \in M_n(\Q [\Z])$. Then, $$L_p(t-a) = L_p(\chi_a(t))= -k \cdot \log(p),$$ where $\chi_a(t)$ denotes the characteristic polynomial of the matrix $a$ and $k$ is the largest natural number so that $p^k$ divides the denominator of some coefficient of $\chi_A(t)$.
\end{lem}

We conjecture that for every positive $f \in \Z \Gamma$, the quantity $L_p(f)$ can be computed locally on $\Gamma$. More precisely,
$$L_p(f) = \lim_F \frac{\log |\det(f_F)|_p}{|F|}$$
in analogy to one of the main results in \cite{lithom}. Here, for any $f \in \Z \Gamma$ and finite subset $F \subset \Gamma$, $f_F$ denotes the $F \times F$-matrix, which is obtained from $f$ by restriction to $\Z F \subset \Z \Gamma$.
A closely related quantity is $F \mapsto \dim_{\Z/p\Z}(\ker(f_F \colon (\Z/p\Z)F \to (\Z/p\Z)F))$. It follows from results of Elek \cite{Elek03a}, that the normalized limit exists in this case.

\section{Torsion submodules and localization}

Let $\Gamma$ be an amenable group, such that $\Z \Gamma$ does not contain zero-divisors. It is well-known that $\Z \Gamma$ satisfies the left and right Ore condition, i.e.\ for non-zero $a,b \in \Z \Gamma$, there exist non-zero $c,c',d,d' \in \Z \Gamma$, such that $ca = db$ and $ac'=bd'$. This was first observed in \cite{tamari} for coefficients not only in $\Z$ but in any integral domain, see also \cite[Example 8.16]{Luck}. Let us give a new argument for some interesting cases -- using the torsion invariant for $\Z \Gamma$-modules.

\begin{proposition}[Tamari]
Let $k$ be a finite field or $k=\Z$ and $\Gamma$ be an amenable group, so that $k \Gamma$ does not contain zero-divisors. Then, the ring $k \Gamma$ satisfies the left and right Ore condition.
\end{proposition}
\begin{proof}
In order to prove the left Ore condition, it is enough to show that for any pair $f,g$ of non-zero elements, the map
$$k \Gamma \oplus k \Gamma \ni (x,y) \mapsto xf-yg \in k \Gamma$$
cannot be injective. However, this is clear since $\rho(k \Gamma \oplus k \Gamma) = 2 \log |k| > \log |k| = \rho(k \Gamma)$. In case $k=\Z$, injectivity of
$$\Z \Gamma \oplus \Z \Gamma \ni (x,y) \mapsto xf-yg \in \Z \Gamma$$
yields an injection $\Z \Gamma \ni y \mapsto yg \in \Z \Gamma/ \Z \Gamma f$, which is absurd since $\rho(\Z \Gamma)= \infty$ and $\rho( \Z \Gamma/ \Z \Gamma f)< \infty$.
This finishes the proof.
\end{proof}

For any $\Z \Gamma$-module $\cM$, we define
$${\rm tor}(\cM) := \{ f \in \cM \mid \exists a \in \Z \Gamma, a \neq 0, af=0\}.$$
Using the Ore condition, it is easy to see that ${\rm tor}(\cM) \subset \cM$ is a $\Z \Gamma$-submodule of $\cM$.
We call a $\Z \Gamma$-module $\cM$ \emph{torsionfree} if for all non-zero $f \in \cM$ and non-zero $g \in \Z \Gamma$, $gf\neq 0$.
Consider the extension
$$0 \to {\rm tor}(\cM) \to \cM \to {\rm f}(\cM) \to 0.$$

We denote by $\cO(\Gamma)$ the Ore localization  \cite[Section 10A]{Lam} of $\Z \Gamma$, which is a skew field. Note that there is no difference between the left and the right Ore localization \cite[Corollary 10.14]{Lam}. We can express the rank of a $\Z \Gamma$-module $\cM$ by the formula:
$${\rm rk}(\cM)=\dim_{\cO(\Gamma)} \left(\cO(\Gamma) \otimes_{\Z \Gamma} \cM\right) \in \{0,1,2,\dots \} \cup \{ \infty \}.$$
The inclusion $\Z \Gamma \subset \cO(\Gamma)$ is
a flat ring extension \cite[Propositions 4.3 and 4.4]{Lam}. Hence, an exact sequence $0 \to \cM' \to \cM \to \cM'' \to 0$ yields ${\rm rk}(\cM) = {\rm rk}(\cM') + {\rm rk}(\cM'')$.

\begin{lem} \label{lemmod}
Let $\Gamma$ be an amenable group such that $\Z \Gamma$ does not contain zero-divisors. Let $\cM$ be a $\Z \Gamma$-module.
\begin{enumerate}
\item[(i)] $\tor(\cM)$ is precisely the kernel of the canonical map $\cM \to \cO(\Gamma) \otimes_{\Z \Gamma} \cM$.
\item[(ii)] The module ${\rm f}(\cM)$ is torsionfree.
\item[(iii)] If $\cM$ is finitely generated, then ${\rm f}(\cM)$ embeds into a finitely generated free $\Z \Gamma$-module.
\item[(iv)]  For any $\cM$, ${\rm tor}(\cM)=\cM$ if and only ${\rm rk}(\cM)=0$.

\end{enumerate}
\end{lem}
\begin{proof}
%(i) Let $f,g \in {\rm tor}(\cM)$ and $a,b \in \Z \Gamma$ non-zero such that $af = bg = 0$. Then, there exist non-zero $c,d \in \Z \Gamma$ such that
%$ca = db$ and $caf = dbg = 0$, thus $ca(f \pm g)=0$. This shows that ${\rm tor}(\cM)$ is an abelian group. Let now $b \in \Z \Gamma$ and $f \in {\rm tor}(\cM)$ with $a \in \Z \Gamma$ non-zero and $af=0$. Again, we find non-zero $c,d \in \Z \Gamma$ such that $db = ca$. Hence $d(bf) = caf = 0$. Hence, $bf \in {\rm tor}(\cM)$. This shows that ${\rm tor}(\cM) \subset \cM$ is indeed a $\Z \Gamma$-submodule.
(i) This is Proposition 0.8.1 in \cite{pmcohn}.
(ii) Let $f \in \cM$ and suppose that there exist a non-zero $a \in \Z \Gamma$, such that $af=0$. Let $f'$ be any lift of $f$ to $\cM$. Then, $af' \in {\rm tor}(\cM)$ and there exists $b \in \Z \Gamma$ non-zero, such that $baf'=0$. Hence $f' \in {\rm tor}(\cM)$ and hence $f=0$. This proves the second claim. (iii) Cohn proved that any finitely generated torsionfree module over a left and right Ore domain embeds into a free module \cite[Corollary 0.8.5]{pmcohn}. (iv) follows from (ii) and (i).
\end{proof}

Statement (ii) in the preceding lemma solves Exercise 19 on page 318 of \cite{Lam}. We can now relate the concept of torsion submodule to finiteness of our numerical torsion invariant. Recall that ${\rm rk}(\cM) \neq 0$ implies $\rho(\cM) = \infty$.

\begin{thm} \label{L-torsion has finite entropy}
Let $\Gamma$ be a countable  amenable group such that $\Z\Gamma$ has no zero-divisors. Let $\cM$ be a finitely generated left $\Z\Gamma$-module. Then
$\tor(\cM)$ is the largest submodule of $\cM$ with finite torsion.
\end{thm}
\begin{proof} Let $\cM'$ be a submodule of $\cM$ with finite torsion $\rho(\cM')$. Let $x\in \cM'$. Then $\rho({\Z\Gamma x})\le \rho({\cM'})<\infty$.
If $x\not\in \tor(\cM)$, then $\Z\Gamma x$ is isomorphic to $\Z\Gamma$ as left $\Z\Gamma$-modules and hence $\rho({\Z\Gamma x})=\infty$.
Thus $x\in \tor(\cM)$. Therefore $\cM'\subseteq \tor(\cM)$.

Next we show $\rho({\tor(\cM)})<\infty$ by induction on the number of generators of $\cM$. Consider first the case $\cM$ is generated by some element $y$. If $y\not \in \tor(\cM)$, then $\tor(\cM)=0$ and hence $\rho({\tor(\cM)})=0$. If $y\in \tor(\cM)$, then $ay=0$ for some nonzero $a\in \Z\Gamma$, and hence $\rho({\cM})\le \rho({\Z\Gamma/\Z\Gamma a})<\infty$. Now suppose that $\rho({\tor(\cM)})<\infty$ for any left $\Z\Gamma$-module
generated by $n$ elements. Let $\cM$ be a left $\Z\Gamma$-module generated by $n+1$ elements $y_1, \dots, y_{n+1}$. Denote by $\cM'$ the submodule of $\cM$ generated by $y_1, \dots, y_n$. Then $\cM/\cM'$ is generated by one element. By induction hypothesis we have $\rho({\tor(\cM')}), \rho({\tor(\cM/\cM')})<\infty$. Note that $\tor(\cM)\cap \cM'=\tor(\cM')$, and the quotient map $\pi:\cM\rightarrow \cM/\cM'$ sends $\tor(\cM)$ into $\tor(\cM/\cM')$. From Yuzvinski\u{\i}'s Addition Formula we get
$$ \rho({\tor(\cM)})=\rho({\tor(\cM')})+\rho({\pi(\tor(\cM))})\le \rho({\tor(\cM')})+\rho({\tor(\cM/\cM')})<\infty.$$
This finishes the induction step.
\end{proof}

\begin{cor}
Let $\cM$ be a finitely generated $\Z \Gamma$-module such that every non-zero submodule $\cN \subset \cM$ has $\rho(\cN) = \infty$. Then, there exists an embedding $\cM \subset \Z \Gamma^{\oplus n}$ for some $n \in \N$.
\end{cor}
\begin{proof} By the previous theorem, we have $\cM={\rm f}(\cM)$. Hence, the claim is implied by Lemma \ref{lemmod} (iii).\end{proof}

\begin{remark}
For $G=\{e\}$, the preceding result just says that the torsion subgroup of a finitely generated abelian group is finite.
\end{remark}

\section{Asphericity of 2-complexes}

Let us finish the article by giving some non-trivial applications of the torsion, as an invariant of $\Z \Gamma$-modules. The applications can also be proved using $\ell^2$-invariants and Hilbert space methods rather than entropy and ergodic theoretic methods.

Now we recall the definition of group of type FL. A group $\Gamma$ is of type FL if it admits a finite resolution over $\Z\Gamma$:
$$0\rightarrow \Z\Gamma^{\oplus n_m}\rightarrow ...\rightarrow\Z\Gamma^{\oplus n_1}\rightarrow \Z\Gamma^{\oplus n_0}\rightarrow \Z\rightarrow 0.$$ If $\Gamma$ is of type FL then $\Gamma$ must be torsion free \cite[Corollary VIII.2.5]{brown}, and in particular it is infinite if it is non-trivial.
The following result is due to Cheeger-Gromov \cite{cg}. Using torsion, we can give a new argument.

\begin{proposition} \label{euler}
The Euler characteristic of a non-trivial amenable group $\Gamma$ of type FL vanishes.
\end{proposition}
\begin{proof}
By the previous remarks, $\Gamma$ must be infinite. Now, let
$\cC_* \to \Z$ be a finite resolution of $\Z$ by finitely generated free $\Z \Gamma$-modules. Then, $ (\Z/2 \Z) \otimes_{\Z}\cC_{*} \to \Z/2 \Z$ is exact. By additivity of torsion, $\chi(\Z) \cdot \log 2=\rho(\Z/2 \Z)=0$. Hence, the Euler characteristic of $\Gamma$ vanishes.
\end{proof}

The following theorem is a particular case of Theorem 3.2 in \cite{bh}. Using the torsion invariant, we can give an elementary proof of this result.

\begin{thm}
Let $\Gamma$ be an amenable group. The following are equivalent:
\begin{enumerate}
\item The group $\Gamma$ admits a finite classifying space of dimension two.
\item There exists a natural number $n$, such that the group $\Gamma$ can be defined with $n$ generators and $n-1$ relations.
\end{enumerate}
\end{thm}
\begin{proof}
If $\Gamma$ has a finite classifying space of dimension two, then there must be one more one-cell than two-cells, since $\chi(\Gamma)=0$ by Proposition \ref{euler}. The converse implication follows from the next lemma.
\end{proof}

\begin{definition} \label{cohodef}
The cohomological dimension of $\Gamma$, denoted by $\mbox{cd}(\Gamma)$, is defined to be the smallest integer $n$ such that there exists a projective resolution $0\to P_n\to ...\to P_0\to \Z\to 0$, where $P_i$ are projective $\Z\Gamma$-modules, if there exist such integers $n$; otherwise we set $\mbox{cd}(\Gamma)=\infty$.
\end{definition}
We also have $\mbox{cd}(\Gamma)=\inf\{n: H^i(\Gamma,M)=0, \mbox{ for } i>n \mbox{ and all } \Z\Gamma\mbox{-modules } M\}$ \cite[Section VIII.2]{brown}.

\begin{lem} \label{stronger}
Let $\Gamma = \langle X| R \rangle$ be an amenable group defined with $n$ generators and $n-1$ relations. Then, the presentation $2$-complex associated with the presentation is aspherical. In particular, the group $\Gamma$ has cohomological dimension two.
\end{lem}
\begin{proof}In order to show that the presentation $2$-complex $X$ is aspherical, it is enough to show that $\pi_2(X)=0$.
Since $X$ has one cell in dimension zero, $n$ cells in dimension one and $n-1$ cells in dimension two, we have $\chi(X)= 1 - n + (n-1)=0$. The first homology of $X$ is computed by a complex of the form
$$0 \to \Z^{n-1} \to \Z^n \to \Z \to 0.$$
Since $H_0(X)=\Z$ and ${\rm rk}_{\Z} H_2(X) \leq n-1$, we conclude from $\chi(X)=0$ that ${\rm rk}_{\Z }H_1(X) \neq 0$. Hence, $\Gamma=\pi_1(X)$ surjects onto $\Z$. We conclude that $\Gamma$ is infinite.

We denote by $\tilde X$ the universal cover of $X$.
Note that $H_1(\tilde X,\Z) = 0$ and by Hurewicz' theorem
$$\pi_2(X) = \pi_2(\tilde X) = H_2(\tilde X,\Z).$$
Let $k$ be a finite field. By the universal coefficient theorem
$$0 \rightarrow H_i(\tilde X, \Z)\otimes_{\Z} k\rightarrow H_i(\tilde X,k)\rightarrow {\rm Tor}(H_{i-1}(\tilde X, \Z),k)\rightarrow 0,$$
so that we conclude that there are natural isomorphisms
\begin{equation} \label{pi=h}
\pi_2(X) \otimes_{\Z} k = H_2(\tilde X,\Z) \otimes_{\Z} k = H_2(\tilde X,k).
\end{equation}
Then, computing the homology of $\tilde X$ with coefficients in $k$ using that $\Gamma$-equivariant cellular chain complex of $\tilde X$, we get an exact sequence
$$0 \to H_2(\tilde X,k) \to (k \Gamma)^{n-1} \to (k \Gamma)^{n} \to k \Gamma \to k \to 0$$
of left $\Z \Gamma$-modules. We conclude that the torsion of $H_2(\tilde X,k)$ as a $\Z\Gamma$-module vanishes. Since, $H_2(\tilde X,k) \subseteq (k \Gamma)^{n-1}$, we conclude from Lemma \ref{posent} and Yuzvinski\u{\i} addition formula that $0=H_2(\tilde X,k) = H_2(\tilde X,\Z) \otimes_{\Z} k$. In particular, $H_2(\tilde X,\Z)$ is a divisible as an abelian group.
Computing the homology of $\tilde X$ with coefficients in $\Z$, we get an exact sequence
$$0 \to H_2(\tilde X,\Z) \to (\Z \Gamma)^{d_2} \to (\Z \Gamma)^{d_1} \to (\Z \Gamma)^{d_0} \to \Z \to 0.$$
Thus, $H_2(\tilde X,\Z)$ is a subgroup of $(\Z\Gamma)^{d_2}$. Since every divisible subgroup of a free abelian group must vanish, we conclude that $H_2(\tilde X,\Z)=0$. This shows that $\pi_2(X)=0$ and hence, $X$ is aspherical.
\end{proof}

\begin{remark}
For the proof of Lemma \ref{stronger}, it is enough to assume that $\beta_1^{(2)}(\Gamma)=0$, see \cite{bh}.
\end{remark}

It is natural to ask which examples of amenable groups of cohomological dimension two exist. We denote by ${\rm rk}(\Gamma)$ the minimal cardinality of generating set of $\Gamma$. In view of the available examples, it is natural to ask the following question:
\begin{question}
\label{Q-deficiency}
Does every amenable group $\Gamma$ of cohomological dimension 2 satisfy ${\rm rk}(\Gamma)\leq 2$?
\end{question}
In the case of elementary amenable groups, Question \ref{Q-deficiency} has an affirmative answer. The class EG of elementary amenable groups is the smallest class of groups containing all finite groups, abelian groups, and is closed under extension and directed unions \cite{Chou}. Clearly, virtually solvable groups are elementary amenable. An alternative description of EG was introduced in \cite{MR964842} as follows. If $\mathcal{X},\mathcal{Y}$ are classes of groups, let $\mathcal{XY}$ denote the class of groups $G$ such that there exists an exact sequence $1\to N\to G\to H\to 1$ with $N$ in $\mathcal{X}$ and $H$ in $\mathcal{Y}$. Let $L\mathcal{X}$ denote the class of groups $G$ such that each finite subset of $G$ is contained in some $\mathcal{X}$-subgroup. Let $\mathcal{X}_1$ the class of finitely generated virtually abelian groups and for any ordinal $\alpha$, define $\mathcal{X}_\alpha$ inductively as follows: $\mathcal{X}_0=\{1\}$, $\mathcal{X}_{\alpha}=(L\mathcal{X}_{\alpha-1})\mathcal{X}_1$ if $\alpha$ is a successor ordinal, and $\mathcal{X_\beta}=\bigcup_{\alpha<\beta}\mathcal{X}_\alpha$ if $\beta$ is a limit ordinal. Then each $\mathcal{X}_\alpha$ is subgroup closed and ${\rm EG}=\bigcup_{\alpha\geq 0}\mathcal{X}_\alpha$ \cite[Lemma 3.1]{MR964842}.

Now we recall briefly the definition of Hirsch length for elementary amenable groups introduced by Hillman \cite{Hillman}. Let $\Gamma$ be an elementary amenable group. Put $o(\Gamma)=\min\{\alpha|\Gamma\in \mathcal{X}_\alpha\}$. If $\Gamma$ is in $\cX_1$, that is it has a finite index finitely generated abelian subgroup $A$, define $h(\Gamma)={\rm rk}_{\Z}A$. Suppose that the Hirsch length has been defined for all groups in $\cX_\alpha$ and $o(\Gamma)=\alpha+1$. If $N\in L\cX_\alpha$, let $h(N):=\sup\{h(F): F \mbox{ is an }\cX_\alpha\mbox{-subgroup of } N\}$. Finally, if $\Gamma$ is in $\cX_{\alpha+1}$ then it has a normal subgroup $N$ in $L\cX_\alpha $ with quotient in $\cX_1$, define $h(\Gamma)=h(N)+h(\Gamma/N)$. Then $h(\Gamma)$ is well defined for any $\Gamma\in {\rm EG}$ \cite[Theorem 1]{Hillman}.
\begin{thm}
Let $\Gamma$ be an elementary amenable group defined with $n$ generators and $n-1$ relations. Then, it must be $\Z$, $\Z^2$ or Baumslag-Solitar groups $BS(1,k)=\langle a,b| bab^{-1}a^{-k}\rangle$. In particular, ${\rm rk}(\Gamma) \leq 2$.
\end{thm}
\begin{proof}
From Lemma \ref{stronger} and Lemma 2 of \cite{Hillman}, one has $h(\Gamma)\leq \mbox{cd}(\Gamma)=2$. Then $\Gamma/T$ is solvable where $T$ is the maximal locally finite normal subgroup of $\Gamma$ \cite[Theorem 2]{Hillman}. On the other hand, since the cohomological dimension is finite, $\Gamma$ is torsion free \cite[Corollary VIII.2.5]{brown} and in particular, $T$ is trivial, and we conclude that $\Gamma$ itself is solvable. Now, Theorem 5 of \cite{MR526863} says that every solvable group of cohomological dimension 2 must be $\Z$, $\Z^2$ or Baumslag-Solitar groups $BS(1,k)=\langle a,b| bab^{-1}a^{-k}\rangle$.
\end{proof}

To the best of our knowledge, there are no known amenable groups which are of type FL and not elementary amenable.

\section*{Acknowledgments}

The second author wants to thank Mark Sapir for a helpful remark on MathOverflow and Mikl\'os Ab\'ert for interesting conversations. Both authors thank Hanfeng Li for letting them include his results in Section \ref{hli} and Lewis Bowen, David Kerr, Simone Virili, and Thomas Ward for helpful comments on a previous version. This research was supported by the European Research Council and the Max-Planck Society.

\end{document}